\documentclass[10pt]{amsart}
\usepackage[margin=0.8in]{geometry}
\newtheorem{theorem}{Theorem}[section]
\newtheorem{lemma}[theorem]{Lemma}

\theoremstyle{definition}

\theoremstyle{remark}
\newtheorem{remark}[theorem]{Remark}

\numberwithin{equation}{section}

\begin{document}

\title{On a Positive decomposition of entropy production functional for the polyatomic BGK model.}

\author{Sa Jun Park, SEOK-BAE YUN }
\address{Department of Mathematics, Sungkyunkwan University, Suwon 440-746, Republic of Korea}

\email{parksajune@skku.edu}
\address{Department of Mathematics, Sungkyunkwan University, Suwon 440-746, Republic of Korea}
\email{sbyun01@skku.edu}
\begin{abstract}
	In this paper, we show that  the entropy production functional for the polyatomic ellipsoidal BGK model can be decomposed into two non-negative parts. Two applications of this property: the $H$-theorem for the polyatomic BGK model and the weak compactness of the polyatomic ellipsoidal relaxation operator, are discussed.
\end{abstract}
\maketitle
\section{Introduction}
In this paper, we obtain a positive decomposition of the entropy production functional of the polyatomic ellipsoidal BGK model, which is a relaxation model of the Boltzmann equation for polyatomic gases \cite{A-L-P-P,A-B-L-P,BGK,BS,Wel}:
\begin{align}\label{PolyESBGK}
\partial_tf+v\cdot\nabla_xf=A_{\nu,\theta}\{\mathcal{M}_{\nu,\theta}(f)-f\}.
\end{align}
The velocity distribution function $f(t,x,v,I)$ denotes the number density of the particles on phase point
$(x,v) \in \Omega_x\times\mathbb{R}^3_v$ at time $t\in\mathbb{R}^+$ with non-translational internal energy $I^{\frac{2}{\delta}}\in\mathbb{R}^+$, where $\delta>0$ is the degree of freedom of non-translational motions of the molecules such as the vibration or rotation.

We define the macroscopic local density $\rho(t,x)$, bulk velocity $U(t,x)$, stress tensor $\Theta(t,x)$ and internal energy $E_{\delta}(t,x)$ by
\begin{align*}
\rho(t,x)&=\int_{\mathbb{R}^3\times\mathbb{R}^+}f(t,x,v,I)dvdI,\cr
U(t,x)&=\frac{1}{\rho}\int_{\mathbb{R}^3\times\mathbb{R}^+}vf(t,x,v,I)dvdI,\cr
\Theta(t,x)&=\frac{1}{\rho}\int_{\mathbb{R}^3\times\mathbb{R}^+}\left\{(v-U(t,x))\otimes(v-U(t,x))\right\}f(t,x,v,I)dvdI, \cr
E_{\delta}(t,x)&=\int_{\mathbb{R}^3\times\mathbb{R}^+}\left\{\frac{1}{2}|v-U(t,x)|^2+I^{\frac{2}{\delta}}\right\}f(t,x,v,I)dvdI.
\end{align*}
Here, the internal energy $E_{\delta}$ is divided into
the energy due to the translational motion, and the energy attributed to the non-translational motion of the particles :
\begin{align*}
E_{tr}&=\int_{\mathbb{R}^3\times\mathbb{R}^+}\frac{1}{2}|v-U(t,x)|^2f(t,x,v,I)dvdI, \cr
E_{I,\delta}&=\int_{\mathbb{R}^3\times\mathbb{R}^+}I^{\frac{2}{\delta}}f(t,x,v,I)dvdI.
\end{align*}
In view of the equi-partition principle, we define the corresponding temperatures $T_{\delta}$, $T_{tr}$ and $T_{I,\delta}$ by:
\begin{align*}
E_{\delta}=\frac{3+\delta}{2}\rho T_{\delta}, \quad E_{tr}=\frac{3}{2}\rho T_{tr}, \quad E_{I,\delta}&=\frac{\delta}{2}\rho T_{I,\delta}.
\end{align*}
The collision frequency $A_{\nu,\theta}$ is defined by $A_{\nu,\theta}=(\rho T_{\delta})/(1-\nu+\nu\theta)$. The relaxation parameters $-1/2<\nu<1$ and $0\leq\theta\leq1$ were introduced so that the
Prandtl number and the second viscosity coefficient are correctly derived through the Chapmann-Enskog expansion. The corrected temperature tensor $\mathcal{T}_{\nu,\theta}$ and the relaxation temperature $T_{\theta}$ are defined by
\begin{align*}
\mathcal{T}_{\nu,\theta}&=(1-\theta)\{(1-\nu)T_{tr}Id+\nu\Theta\}+\theta T_{\delta} Id, \cr
T_{\theta}&=(1-\theta)T_{I,\delta}+\theta T_{\delta}.
\end{align*}
The polyatomic Gaussian $\mathcal{M}_{\nu,\theta}$ is given by
\begin{align*}
\mathcal{M}_{\nu,\theta}(f)=\frac{\rho\Lambda_{\delta}}{\sqrt{\det(2\pi\mathcal{T}_{\nu,\theta})}T_{\theta}^{\frac{\delta}{2}}}\exp\left(-\frac{1}{2}(v-U)^{\top}\mathcal{T}_{\nu,\theta}^{-1}(v-U)-\frac{I^{\frac{2}{\delta}}}{T_{\theta}}\right).
\end{align*}	
Here, $\Lambda_{\delta}$ is the normalizing factor: $\Lambda_{\delta}=1/\int_{\mathbb{R}^+}e^{-I^{\frac{2}{\delta}}}dI$.
This relaxation operator satisfies
\begin{align*}
\int_{\mathbb{R}^3\times\mathbb{R}^+}\{\mathcal{M}_{\nu,\theta}(f)(t,x,v,I)-f(t,x,v,I)\}
\left(\begin{array}{ccc}
1 \cr v \cr \frac{1}{2}|v|^2+I^{\frac{2}{\delta}}
\end{array}\right) dvdI=0,
\end{align*}
which implies the conservation of mass, momentum and energy:
The $H$-theorem for this model was verified first by Andries et al in \cite{A-L-P-P} (See also \cite{BS,PY}):
\begin{align*}
\frac{d}{dt}\int_{\mathbb{R}^3\times\mathbb{R}^+}f(t)\ln f(t)dvdI \leq 0.
\end{align*}
For many kinetic equations, such as the Boltzmann equation or the original BGK model, the non-negativity of the entropy production functional is the result of the following elementary inequality:
\begin{align}\label{WK}
\mathcal{E}(A,B)=(A-B)(\ln A-\ln B)\geq0.
\end{align}
Entropy production functional $D_{\nu,\theta}(f)$ for our model, however, has additional term besides this:
\begin{align}\label{Dtheta}
\begin{split}
D_{\nu,\theta}(f)&=
-\int_{\mathbb{R}^3 \times \mathbb{R}^+}\{\mathcal{M}_{\nu,\theta}(f)-f\}\ln fdvdI \cr
&=\int_{\mathbb{R}^3 \times \mathbb{R}^+}\{\mathcal{M}_{\nu,\theta}(f)-f\}\{\ln\mathcal{M}_{\nu,\theta}(f)-\ln f\}dvdI +R_{\nu,\theta}
\end{split}
\end{align}
where the remainder term $R_{\nu,\theta}$ is defined by
\begin{align*}
R_{\nu,\theta}=\int_{\mathbb{R}^3 \times \mathbb{R}^+}\{\mathcal{M}_{\nu,\theta}(f)-f\}\left(\frac{1}{2}(v-U)^{\top}\mathcal{T}_{\nu,\theta}^{-1}(v-U)+\frac{I^{\frac{2}{\delta}}}{T_{\theta}}\right)dvdI.
\end{align*}
The first term of the second line in (\ref{Dtheta}) is clearly non-negative due to (\ref{WK}). The following theorem says that the second term is also non-negative if $\nu$ is non-negative:
\begin{theorem}\label{theorem} Let $0\leq \theta \leq 1$ and $0\leq\nu<1$. Then the remainder term $R_{\nu,\theta}$ is non-negative:
\begin{align*}
R_{\nu,\theta}\geq0.
\end{align*}
\end{theorem}
\begin{remark}
\noindent (1) This positive decomposition of the entropy production functional can be used to prove the $H$-theorem, and the $L^1$ compactness of the polyatomic Gaussian. (See Section 4.)\newline
\noindent (2) For the monatomic ellipsoidal BGK model, a rather complete result is available \cite{Yun2}: The remainder term is positive when $\nu>0$, zero when $\nu=0$, and negative if $-1/2<\nu<0$ . For the polyatomic model, the last case  is inconclusive . We leave it to future work.\newline
(3) This result is a priori estimate. For this to be rigorous, we should start from existence theory. The existence of solutions for which all the above integral and computations make sense was obtained recently in \cite{Yun}.
\end{remark}

This paper is organized as follows.
In Section 2, several useful lemmas are established. The proof of the main theorem is given in Section 3. In Section 4, two applications of this result: the $H$-theorem and the weak compactness of polyatomic Gaussian, are discussed.
\section{Lemmas}
In this section, we establish lemmas which are crucially used in the proof later. We start with a reformulation of the remainder term.
\begin{lemma}\label{rewritten}
We can write $R_{\nu,\theta}$ as
\begin{align}\label{R}
R_{\nu,\theta}=\frac{\rho}{2}\left\{3+\delta-\left(F_{\theta}+\frac{\delta T_{I,\delta}}{T_{\theta}}\right)\right\},
\end{align}
where $F_{\theta}$ is given by
\begin{align*}
F_{\theta}
&=\sum^3_{i=1}\frac{\Theta_i}{A_i}
\end{align*}
for  the eigenvalues $\Theta_i~(i=1,2,3)$ of $\Theta$ and $A_i=(1-\theta)\{(1-\nu)T_{tr}+\nu\Theta_i\}+\theta T_{\delta }$.
\end{lemma}
\begin{proof}
Using the identity:
\[
X^{\top}\{Y\}X=\{X\otimes X\}:Y,\qquad( X\in \mathbb{R}^n, Y\in\mathbb{R}^{n\times n})
\]
where $X:Y$ denotes the Frobenius product:
$A:B=\sum_{i,j}A_{ij}B_{ij}$,
we have
\begin{align*}
\begin{split}
R_{\nu,\theta} &=\frac{1}{2}\left\{\int_{\mathbb{R}^3 \times \mathbb{R}^+}\{\mathcal{M}_{\nu,\theta}(f)-f \}(v-U)\otimes (v-U)dvdI \right\}:\mathcal{T}_{\nu,\theta}^{-1} +\frac{\rho}{2}\frac{\delta(T_{\theta}-T_{I,\delta})}{T_{\theta}} \cr
&=\frac{\rho}{2}\{\mathcal{T}_{\nu,\theta}-\Theta\} :\mathcal{T}_{\nu,\theta}^{-1}+\frac{\rho}{2}\frac{\delta(T_{\theta}-T_{I,\delta})}{T_{\theta}}.		
\end{split}
\end{align*}
Then, recalling the identity
\[
A:B=tr\big(A^{\top}B\big),
\]
we compute
\begin{align}\label{R}
\begin{split}
R_{\nu,\theta}=\frac{\rho}{2}\{3-tr\big(\Theta^{\top}\mathcal{T}_{\nu,\theta}^{-1}\big) \}+\frac{\rho}{2}\frac{\delta(T_{\theta}-T_{I,\delta})}{T_{\theta}}
=\frac{\rho}{2}\left\{3+\delta-\left(tr\big(\Theta^{\top}\mathcal{T}_{\nu,\theta}^{-1}\big) +\frac{\delta T_{I,\delta}}{T_{\theta}} \right) \right\}.
\end{split}
\end{align}
Since $\Theta$ is symmetric, there exists an orthogonal matrix $P$ such that
\begin{align}\label{Diagtheta}
P^{\top}\Theta P=diag\{\Theta_1,\Theta_2,\Theta_3\}.
\end{align}
Here, $diag\{a,b,\cdots\}$ denotes a diagonal matrix with diagonal entries $a,b,\cdots$.
Using the same $P$, we can also diagonalize $\mathcal{T}_{\nu,\theta}$ as
\begin{align*}
P^{\top}\mathcal{T}_{\nu,\theta}P
&=
P^{\top}[\,(1-\theta)\{(1-\nu)T_{tr}Id+\nu\Theta\}+\theta T_{\delta}Id\,]P \cr
&=(1-\theta)\{(1-\nu)T_{tr}Id+\nu P^{\top}\Theta P\}+\theta T_{\delta}Id\cr
&=(1-\theta)\big[(1-\nu)T_{tr}Id+\nu diag\{\Theta_1,\Theta_2,\Theta_3\}\big]+\theta T_{\delta}Id\cr
&=diag\big\{
A_1,A_2,A_3
\big\}.
\end{align*}
Therefore, we have
\begin{align*}
P^{\top}\mathcal{T}_{\nu,\theta}^{-1}P=\{P^{\top}\mathcal{T}_{\nu,\theta}P\}^{-1}
=diag\big\{A^{-1}_1, A^{-1}_2, A^{-1}_3\big\}
\end{align*}
so that, from the similarity invariance of the trace operator, we get
\begin{align*}
tr\big(\Theta^{\top}\mathcal{T}_{\nu,\theta}^{-1}\big)
= tr\big(P^{\top}\Theta^{\top}\mathcal{T}_{\nu,\theta}^{-1}P\big)
= tr\big(\{P^{\top}\Theta^{\top}P\}\{P^{\top}\mathcal{T}_{\nu,\theta}^{-1}P\}\big)
	=\sum_{i=1,2,3}\frac{\Theta_i}{A_i}.
\end{align*}
\end{proof}
The above lemma shows that a proper estimate of $F_{\theta}$ is important, which is given in the following lemma.
\begin{lemma}\label{F theta} Let $0\leq\nu<1$ and $0\leq \theta\leq 1$. Then we have
\begin{align*}
F_{\theta}\leq \frac{3T_{tr}}{(1-\theta)T_{tr}+\theta T_{\delta }}.
\end{align*}
\end{lemma}
\begin{proof}
{\bf Case 1: ($0<\nu<1$)}
Recall the definition of $A_i$ to see that we can write
\begin{align}\label{defABC}
\Theta_i=\{(1-\theta)\nu\}^{-1}\big\{A_i-(1-\theta)(1-\nu)T_{tr}-\theta T_{\delta}\big\}, \quad(i=1,2,3).
\end{align}
Then, $F_{\theta}$, defined in Lemma \ref{rewritten}, is rewritten as follows:
\begin{align}\label{F}
\begin{split}
F_{\theta}&=\sum_{i=1}^3\frac{\Theta_i}{A_i} \cr
&=\sum_{i=1}^3\frac{\{(1-\theta)\nu\}^{-1}\big\{A_i-(1-\theta)(1-\nu)T_{tr}-\theta T_{\delta}\big\}}{A_i}\cr
&=\frac{3}{(1-\theta)\nu}-\left(\frac{(1-\theta)(1-\nu)T_{tr}+\theta T_{\delta}}{(1-\theta)\nu}\right)\left(\frac{1}{A_1}+\frac{1}{A_2}+\frac{1}{A_3}\right).
\end{split}
\end{align}
Now, by Jensen's inequality,
\begin{align*}
\frac{1}{A_1}+\frac{1}{A_2}+\frac{1}{A_3} &\geq \frac{3}{\frac{A_1+A_2+A_3}{3}}\cr &=
\frac{9}{(1-\theta)\{(1-\nu)3T_{tr}+\nu(\Theta_1+\Theta_2+\Theta_3)\}+3\theta T_{\delta}} \cr
&=
\frac{3}{(1-\theta)T_{tr}+\theta T_{\delta}},
\end{align*}
where we used $\Theta_1+\Theta_2+\Theta_3=3T_{tr}$.  Therefore, we have for $\nu>0$,
\begin{align*}
F_{\theta}
&\leq \frac{3}{(1-\theta)\nu}-\left(\frac{(1-\theta)(1-\nu)T_{tr}+\theta T_{\delta} }{(1-\theta)\nu} \right)\frac{3}{(1-\theta)T_{tr}+\theta T_{\delta}} \cr
&= \frac{3}{(1-\theta)\nu}-\left(\frac{(1-\theta)T_{tr}+\theta T_{\delta} }{(1-\theta)\nu}-\frac{(1-\theta)\nu T_{tr}}{(1-\theta)\nu} \right)\frac{3}{(1-\theta)T_{tr}+\theta T_{\delta}}\cr
&=
\frac{3}{(1-\theta)\nu}-\frac{3}{(1-\theta)\nu}+\frac{3T_{tr}}{(1-\theta)T_{tr}+\theta T_{\delta}} \cr
&=\frac{3T_{tr}}{(1-\theta)T_{tr}+\theta T_{\delta}}.
\end{align*}
{\bf Case 2: ($\nu=0$)} In this case, the $\Theta_i$ in the denominator vanishes and the computation is simplified a lot:
\[
F_{\theta}=\frac{\Theta_1+\Theta_2+\Theta_3}{(1-\theta)T_{tr}+\theta T_{\delta}}
=\frac{3T_{tr}}{(1-\theta)T_{tr}+\theta T_{\delta}}.
\]
\end{proof}
The following convexity property plays an important role in the proof of main theorem.
\begin{lemma}\label{convex} Let $A,B\geq 0$. Define
\[
K=\frac{3}{3+\delta}A+\frac{\delta }{3+\delta}B.
\]
Then, for $0\leq t\leq 1$, we have
\begin{align*}
\frac{3A}{(1-t)A+tK}+\frac{\delta B}{(1-t)B+tK} \leq 3+\delta.
\end{align*}
\end{lemma}
\begin{proof}
Set
\[
F(t)=\frac{3A}{(1-t)A+tK}+\frac{\delta B}{(1-t)B+tK}.
\]
Then, it is easy to see that  $F(0)=F(1)= 3+\delta$.
Therefore, the desired result follows once it is verified that $F(t)$ is convex  on $t\in[0,1]$. First, we write
\begin{align*}
F(t)
&=\frac{3A}{(1-t)A+t\left(\frac{3}{3+\delta}A+\frac{\delta}{3+\delta}B\right)}+\frac{\delta B}{(1-t)B+t\left(\frac{3}{3+\delta}A+\frac{\delta}{3+\delta}B \right)} \cr
&=\frac{3(3+\delta)A}{(3+\delta)A-\delta t(A-B)}+\frac{\delta(3+\delta)B}{(3+\delta)B+3t(A-B)} \cr
&=(3+\delta)\bigg\{ \frac{3A}{(3+\delta)A-\delta t(A-B)}+\frac{\delta B}{(3+\delta)B+3t(A-B)}\bigg\}.
\end{align*}
Then, a straightforward computation gives
\begin{align*}
\frac{1}{3+\delta}\frac{d}{dt}F(t)= \frac{3\delta A(A-B)}{[(3+\delta)A-\delta t(A-B)]^2}-\frac{3\delta B(A-B)}{[(3+\delta)B+3t(A-B)]^2}.
\end{align*}
Likewise,
\begin{align*}
\frac{1}{3+\delta}\frac{d^2}{dt^2}F(t)&= \frac{d}{dt}\left[\frac{3\delta A(A-B)}{\big\{(3+\delta)A-\delta t(A-B)\big\}^2}-\frac{3\delta B(A-B)}{\big\{(3+\delta)B+3t(A-B)\big\}^2}\right] \cr
&=\frac{3\delta A(A-B)2[(3+\delta)A-\delta t(A-B)]\delta(A-B)}{[(3+\delta)A-\delta t(A-B)]^4}\cr &+\frac{3\delta B(A-B)2\big\{(3+\delta)B+3t(A-B)\big\}3(A-B) }{\big\{(3+\delta)B+3t(A-B)\big\}^4} \cr
&=\frac{3\delta^2A(A-B)^2}{\big\{(3+\delta)A-\delta t(A-B)\big\}^3}
+\frac{18\delta B(A-B)^2}{\big\{(3+\delta)B+3t(A-B)\big\}^3} \cr
&=\frac{3\delta^2A(A-B)^2}{\big\{3A+\delta tB+\delta(1-t)A \big\}^3}
+\frac{18\delta B(A-B)^2}{\big\{\delta B+3tA+3(1-t)B \big\}^3}\cr
&\geq0,
\end{align*}
for $t\in [0,1]$. This completes the proof.
\end{proof}

\section{Proof of main theorem}
Now we are ready to prove the main theorem. In view of Lemma \ref{rewritten},
our goal reduces to proving
\begin{align}\label{following}
F_{\theta}+\frac{\delta T_{I,\delta}}{T_\theta}\leq3+\delta.
\end{align}

We divide the proof into the following two cases.\newline

\noindent{\bf(1) The case $0<\theta\leq1$:} Recalling Lemma \ref{F theta} and the definition of $T_{\theta}$, we see that
\begin{align*}
F_{\theta}+\frac{\delta T_{I,\delta}}{T_\theta}
\leq\frac{3T_{tr}}{(1-\theta)T_{tr}+\theta T_{\delta}}+\frac{\delta T_{I,\delta}}{(1-\theta)T_{I,\delta}+\theta T_{\delta}}.
\end{align*}
We then apply the result of Lemma \ref{convex} to bound this by $3+\delta$:
\[
F_{\theta}+\frac{\delta T_{I,\delta}}{T_\theta}\leq 3+\delta.
\]
Now the conclusion follows directly from this and Lemma \ref{rewritten}:
\begin{align*}
R_{\nu,\theta}=\frac{\rho}{2}\left\{3+\delta-\left(F_{\theta}+\frac{\delta T_{I,\delta}}{T_{\theta}} \right) \right\}
\geq \frac{\rho}{2}\left\{3+\delta-(3+\delta) \right\}
\geq 0.
\end{align*}

\noindent {\bf(2) The case of $\theta=0$:} In this case, $R_{\nu,0}$ reduces to the remainder term in the entropy dissipation of the monatomic ellipsoidal BGK model:
\section{Applications}
The positive decomposition of $D_{\nu,\theta}(f)$ can be used to derive  the $H$-theorem, and the weak compactness of Polyatomic Gaussian $\mathcal{M}_{\nu,\theta}(f^n)$.
\subsection{$H$-theorem}
By Theorem 1.1, (\ref{WK}) and (\ref{Dtheta}), we see that
\begin{align*}
\frac{d}{dt}\int_{\mathbb{R}^3\times\mathbb{R}^+}f\ln f dvdI
&=-\int_{\mathbb{R}^3 \times \mathbb{R}^+}\{\mathcal{M}_{\nu,\theta}(f)-f\}(\ln\mathcal{M}_{\nu,\theta}(f)-\ln f)dvdI -R_{\nu,\theta} \cr
&\leq -\int_{\mathbb{R}^3 \times \mathbb{R}^+}\{\mathcal{M}_{\nu,\theta}(f)-f\}(\ln\mathcal{M}_{\nu,\theta}(f)-\ln f)dvdI\leq0,
\end{align*}
which gives the desired result.
\subsection{Weak compactness of $\mathcal{M}_{\nu,\theta}(f^n)$}
Suppose that the sequence $\{f^n\}$ is weakly compact in $L^1$. Then for $M>0$, we deduce
\begin{align*}
\mathcal{M}_{\nu,\theta}(f^n)-f^n
&=
\{\mathcal{M}_{\nu,\theta}(f^n)-f^n\}(1_{\mathcal{M}_{\nu,\theta}(f^n)< Mf^n}+1_{\mathcal{M}_{\nu,\theta}(f^n)\geq Mf^n}) \cr
&\leq
(M-1)f^n1_{\mathcal{M}_{\nu,\theta}(f^n)< Mf^n} \cr
&+
\frac{1}{\ln M}(\mathcal{M}_{\nu,\theta}(f^n)-f^n)(\ln\mathcal{M}_{\nu,\theta}(f^n)-\ln f^n)1_{\mathcal{M}_{\nu,\theta}(f^n)\geq Mf^n}.
\end{align*}
Hence we have
\begin{align*}
\mathcal{M}_{\nu,\theta}(f^n) \leq Mf^n+\frac{1}{\ln M}(\mathcal{M}_{\nu,\theta}(f^n)-f^n)(\ln\mathcal{M}_{\nu,\theta}(f^n)-\ln f^n).
\end{align*}
Now, take a measurable set $B\subset \mathbb{T}^3_x\times\mathbb{R}^3_v\times\mathbb{R}^+_I$ and  $T\in\mathbb{R}^+$, and use the
non-negativity of $R_{\nu,\theta}$ to compute
\begin{align*}
&\int_{0}^{T}\int_{B}\mathcal{M}_{\nu,\theta}(f^n)dxdvdIdt \cr
&\leq
M\int_{0}^{T}\int_{B}f^ndxdvdIdt
+
\frac{1}{\ln M}\int_{0}^{T}\int_{\mathbb{T}^3_x\times\mathbb{R}^3_v\times\mathbb{R}^+}(\mathcal{M}_{\nu,\theta}(f^n)-f^n)(\ln\mathcal{M}_{\nu,\theta}(f^n)-\ln f^n)dxdvdIdt \cr
&=
M\int_{0}^{T}\int_{B}f^ndxdvdIdt +
\frac{1}{\ln M}\int_{0}^{T}\int_{\mathbb{T}^3_x}D_{\nu,\theta}(f^n)dxdt-\frac{1}{\ln M}\int_{0}^{T}\int_{\mathbb{T}^3_x}R_{\nu,\theta}dxdt \cr
&\leq
M\int_{0}^{T}\int_{B}f^ndxdvdIdt +
\frac{1}{\ln M}\int_{0}^{T}\int_{\mathbb{T}^3_x}D_{\nu,\theta}(f^n)dxdt,
\end{align*}
which gives
\begin{align*}
\int_{0}^{T}\int_{B}\mathcal{M}_{\nu,\theta}(f^n)dxdvdIdt
&\leq
M\int_{0}^{T}\int_{B}f^ndxdvdIdt
\cr
&+
\frac{1}{\ln M}\left(\int_{\mathbb{T}^3_x\times\mathbb{R}^3_v\times\mathbb{R}^+}f_0|\ln f_0|dxdvdI+C(f_0,T)\right),
\end{align*}
This implies that $\mathcal{M}_{\nu,\theta}(f^n)$ is weakly compact in $L^1$ by Dunford-Pettis theorem.\newline

\noindent{\bf Acknowledgement}
This research was supported by Basic Science Research Program through the National Research Foundation of Korea(NRF) funded by the Ministry of Education(NRF-2016R1D1A1B03935955)

\bibliographystyle{amsplain}


\end{document}